\documentclass{article}
\usepackage{graphicx} 
\usepackage{amsmath}
\usepackage{amsthm}
\usepackage{booktabs}
\usepackage{amssymb}
\usepackage{mathtools}
\usepackage{color}
\usepackage{authblk}
\usepackage[verbose,tmargin=2cm,bmargin=2cm,lmargin=2cm,rmargin=2cm]{geometry}
\usepackage{OldStandard}
\newtheorem*{theorem}{Theorem}
\newtheorem*{remark}{Remark}
\newtheorem*{example}{Example}
\newtheorem*{ex}{Example}
\newtheorem*{proposition*}{Proposition}
\newtheorem*{definition*}{Definition}

\newtheorem{prop}{Proposition}
\newtheorem{lemma}{Lemma}
\newtheorem{proposition}{Proposition}

\DeclareMathOperator{\sech}{sech}
\title{The joy of multisection:  elementary approaches to Ramanujan's  lacunary identities for Bernoulli numbers}
\author[1]{Parth Chavan}
\author[2]{Christophe Vignat}
\affil[1]{\'Ecole Polytechnique, Palaiseau, France}
\affil[2]{Department of Mathematics, Tulane University, New Orleans,  USA}
\affil[2]{L2S, CentraleSupelec, Université Paris-Saclay, France}
\date{}
\makeatletter
\newcommand{\subjclass}[2][2020]{%
  \let\@oldtitle\@title%
  \gdef\@title{\@oldtitle\footnotetext{#1 \emph{Mathematics subject classification.} #2}}%
}
\newcommand{\keywords}[1]{%
  \let\@@oldtitle\@title%
\gdef\@title{\@@oldtitle\footnotetext{\emph{Key words:} #1.}}%
}
\makeatother
\begin{document}

\maketitle
\begin{abstract}
An identity by  Ramanujan related to the multisection of Bernoulli numbers is revisited. Two alternative approaches are proposed, both relying on the multisection technique. A geometric approach reveals the role played by the symmetries of the summation domain in the complex plane induced by the multisection technique. The second approach, based on generating functions, allows us to extend Ramanujan's identity to other special functions such as Eisenstein's series.
\end{abstract}

\section{Introduction}
The five volumes  edition of Ramanujan's Notebooks, together  with the additional five volumes of Ramanujan's Lost Notebooks, include many puzzling entries. Here is one of them. Entry 33 in Chapter 23 of the fourth volume of Ramanujan's Notebooks \cite{Berndt} states:
\begin{proposition*}
Let $a_n$ be defined by 
\[
a_n = (2+\sqrt{3})^n+(2-\sqrt{3})^n+2^n,
\]
and let $B_n$ denote the $n-$th Bernouli number, where $n \ge 0.$ Then, for $\vert  x \vert <2\pi,$
\begin{equation}
\label{convolution}
6 \sum_{j\ge 0}B_{12j}\frac{x^{12j}}{(12j)!}
\sum_{k\ge  0}(-1)^k a_{6k+3}\frac{x^{12k+6}}{(12k+6)!}
=
\sum_{k\ge  0}(-1)^k a_{6k+3}\frac{x^{12k+6}}{(12k+5)!}.
\end{equation}
\end{proposition*}

As noticed in \cite{Berndt}, identifying the coefficient of $x^{12n+6}$ in the Taylor expansion of both sides allows us to rephrase this identity as follows:
\begin{equation}
\label{finite}
6\sum_{k = 0}^n \frac{(-1)^k a_{6k+3}}{(12k+6)!}\frac{B_{12n-12k}}{(12n-12k)!}  = \frac{(-1)^n a_{6n+3}}{(12n+5)!},\,\,n\ge 0 .
\end{equation}

Let us recall that the Bernoulli polynomials are defined by the exponential generating function
\begin{align}
\label{Bernoulli}
    \frac{x e^{xz}}{e^x-1} = \sum_{n=0}^{\infty} B_n(z) \frac{x^{n}}{n!}, \quad |x|< 2\pi,
\end{align}
and the Bernoulli numbers are their values at $z=0,$
\[
B_n = B_n(0).
\]
The intriguing fact about identity \eqref{convolution} is that the  Bernoulli generating function on the left-hand side  vanishes after multiplication by the generating function of the sequence $\{a_n\}$ that reappears, barely unchanged, on the right hand side: it looks like that the latter has absorbed the former.

Identity \eqref{finite} is a lacunary identity: it allows to compute Bernoulli numbers using a recursion with gaps of length 12.
As for every other entry in the Notebooks, a proof and some context for this entry are produced by B. Berndt in \cite{Berndt}. The proof uses a general result by Lehmer \cite{Lehmer} based  on the properties of a certain recursive sequence. 

Our main goal here is to provide two other angles under which this identity can be approached:
we propose two different proofs of identity \eqref{convolution} that bypass Lehmer's result and produce some intuition on the mechanism underlying Ramanujan's identity. 

The first approach, the geometric approach, is based on the interpretation of series such as in  \eqref{convolution}
as sums over lattices in the complex plane. Studying these lattices allows a  better understanding of  \eqref{convolution}. The second approach uses generating functions: as usual with generating functions, the loss of combinatorial understanding  is compensated by a  capability of generalization that allows us to produce new extensions of Ramanujan's lacunary identity, for example to the case of Eisenstein series.

The following well-known multisection  formula \cite[Ch.4]{Riordan}  will be repeatedly used in the sequel :
\begin{proposition*}
For an analytic function $f(z)=\sum_{n \ge 0}f_n z^n$, with $p,q$ integers such that  $0\le p<q$ and   $\omega=e^{\imath \frac{2 \pi}{q}},$
\begin{equation}
\label{general multisection}
\frac{1}{q}\sum_{k=0}^{q-1} \frac{1}{\omega^{kp}} f\left(z\omega^k\right) = \sum_{n=0}^{\infty} f_{qn+p} z^{qn+p}.
\end{equation}
\end{proposition*}
This multisection formula is well-known and has multiple applications, for example in combinatorics. Let us quote here Riordan \cite{Riordan} refering to Lehmer \cite{Lehmer}:
\begin{quote}
\texttt{Multisection is not limited to ordinary power series, of course. Perhaps the most interesting illustration of its use with other series is in the determination of lacunary recurrence formulas for the Bernoulli numbers, the history of which is recounted in the paper mentioned in the introduction by Lehmer (1935), who, however, did not use the simple procedure of multisection.} 
\end{quote}
\section{The geometric approach}

\subsection{First step: enter the Bernoulli polynomials}
The following identity is the first step toward a geometric interpretation and a generalization of  identity \eqref{finite}.
\begin{proposition}
\label{Prop_main}
With $p,q$ integers such that $0\le p <q,$ with $\omega = e^{\imath \frac{\pi}{q}}$ and $z\in \mathbb{C},$
\[
\sum_{k=0}^n (-1)^k \binom{2qn+p+q}{2qk+q} z^{2qk+p} B_{2qn+p-2qk}
= \frac{\imath}{2q}\sum_{j=0}^{2q-1}(-1)^j
B_{2qn+p+q}\left(\frac{z}{\omega^{j+\frac{1}{2}}}\right).
\]
\end{proposition}
\begin{proof}
Consider the series 
\[
\sum_{k=0}^n (-1)^k \binom{2qn+p+q}{2qk+q} z^{2qk+q} B_{2qn+p-2qk}
\]
and change the  summation index $k$ to $n-k$ to  obtain
\[
(-1)^n \sum_{k=0}^n (-1)^k \binom{2qn+p+q}{2qk+p} z^{2qn-2qk+q} B_{2qk+p}
=(-1)^n z^{2qn+q}\sum_{k=0}^n \omega^{qk} \binom{2qn+p+q}{2qk+p} z^{-2qk} B_{2qk+p}
\]
\[
=(-1)^n \frac{z^{2qn+p+q}}{\omega^{\frac{p}{2}}}\sum_{k=0}^n (\frac{\sqrt{\omega}}{z})^{2qk+p} \binom{2qn+p+q}{2qk+p}  B_{2qk+p}
\]
The sum is recognized as a $2q-$ multisection and we deduce, using \eqref{general multisection} and the  symbolic notation $B_n(x)=(x+B)^n,$ that it is equal to
\[
(-1)^n \frac{z^{2qn+p+q}}{\omega^{\frac{p}{2}}}\frac{1}{2q}\sum_{j=0}^{2q-1}\frac{1}{\omega^{pj}} (1+\omega^j \frac{\sqrt{\omega}}{z}B)^{2qn+p+q}
=(-1)^n \frac{1}{2q}\sum_{j=0}^{2q-1}\omega^{(j+\frac{1}{2})(2qn+p+q)}
B_{2qn+p+q}(\frac{z}{\omega^{j+\frac{1}{2}}}).
\]
The final result is deduced  from the elementary computation
\[
\omega^{(j+\frac{1}{2})(2qn+p+q)} = \imath (-1)^{n+j}.
\]
\end{proof}
\subsection{Second step: enter the lattices}
The fact that Bernoulli numbers vanish after summation in \eqref{finite} is a consequence of the well-known shift property of the Bernoulli polynomials
\begin{equation}
\label{shift_Bernoulli}
B_n(z+1) = B_n(z) +nz^{n-1},\,\,n\ge0.    
\end{equation}

As a consequence, any  sum that involves pairwise differences of Bernoulli polynomials evaluated at points horizontally separated by one or several unit lengths in the complex plane will result in  a sum of monomials only. This justifies the introduction of the following definition.

\begin{definition*}
A  lattice of $2N$ points  in  the complex plane is   \texttt{integral} if it is made of $N$ pairs of distinct points $\left\{z_i,z_j \right\}$ such as $z_i$ and $z_j$ are horizontally separated by a non-zero integer distance.
\end{definition*}
Notice that we use here the term lattice to designate a finite set of points in the complex plane rather than, as  is more usual, a repeating arrangement of points.

For example, the lattice of four points $\left\{\pm  1 \pm \imath\right\}$ is \texttt{integral} since in both couples of points $\left\{-1-\imath,1-\imath\right\}$ and $\left\{-1+\imath,1+\imath\right\},$ the points are horizontally separated by a distance $D=2.$ 

As another example, with $\omega=e^{\imath\frac{\pi}{6}},$ the lattice  made of the 36 points $\left\{\omega^{11j}+\omega^{11j+11}\right\}$,
$\left\{\omega^{11j+1}+\omega^{11j+10}\right\}$ and $\left\{\omega^{11j+2}+\omega^{11j+9}\right\}$ for $0\le j\le11,$ is an integral lattice, as will be seen in Subsection \ref{233}.
\subsection{Examples}
This subsection introduces some examples of Ramanujan-type series associated with integral lattices.
\subsubsection{the 2-interval formula}
\begin{figure}[h]
\centering
\includegraphics[scale=0.5]{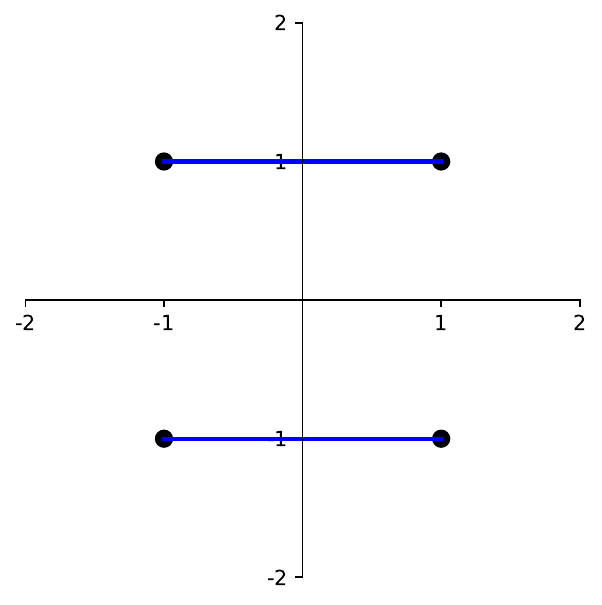} 
\caption{The 2-interval lattice.}
\label{Fig1}
\end{figure}

\begin{proposition*}
    We have the identity for $n\ge 0,$
\[    
\sum_{k=0}^{n}\left(-1\right)^{k}\binom{4n+2}{4k+2} 2^{2k+1}B_{4n-4k} = (-1)^n(2n+1)(2^{2n}+ (-1)^n).
\]
\end{proposition*}
\begin{proof}
Using Proposition \ref{Prop_main}, the series on the left-hand side can be written as 
\begin{align}
\label{2-int}
\frac{i}{4} \left[B_{4n+2}\left(\frac{z}{\imath^{1/2}}\right) - B_{4n+2}\left(\frac{z}{\imath^{3/2}}\right)+ B_{4n+2}\left(\frac{z}{\imath^{5/2}}\right) - B_{4n+2}\left(\frac{z}
{i^{7/2}}\right)\right].
\end{align}
Now notice that $i^{\frac{1}{2}}(1-i) = \sqrt{2}$. Thus, substituting $z= \sqrt{2}$ and using \eqref{shift_Bernoulli},  the sum simplifies to a sum over the integral lattice shown in Figure \ref{Fig1},  equal to
\begin{align*}
    &\frac{i}{4} \left[B_{4n+2}\left(1-i\right) - B_{4n+2}\left(-1-i\right)+ B_{4n+2}\left(-1+i\right) - B_{4n+2}\left(1+i\right)\right] \\
    &= \frac{i}{4}(4n+2) \left[(-1-i)^{4n+1}+(-i)^{4n+1}-(i-1)^{4n+1} -(i)^{4n+1}\right] \\
    &= (2n+1)(1 + (2i)^{2n}) = (-1)^n(2n+1)(2^{2n}+ (-1)^n).
\end{align*}

\end{proof}
\begin{remark}
Any other value of the parameter $z=\frac{q}{\sqrt{2}},\,\,q \in \mathbb{Z}^*,$  retains the property of the lattice being integral and  can be chosen in \eqref{2-int}. For  example, the choice $z=\frac{1}{\sqrt{2}}$ produces the identity
\[
\sum_{k=0}^{n}\left(-1\right)^{k}\binom{4n+2}{4k+2} \frac{1}{2^{2k+1}}B_{4n-4k} = 
(-1)^{n+1}\frac{2n+1}{2^{2n+1}},
\]
while the choice $z=\frac{3}{\sqrt{2}}$ produces 
\begin{equation}
\sum_{k=0}^{n}
(-1)^k\binom{4n+2}{\,4k+2\,}
\left(\frac{9}{2}\right)^{\!2k+1}
B_{4n-4k} = 
(2n+1)\left(\sqrt{10}\left(\frac{5}{2}\right)^{\!2n}
\sin\!\big((4n+1)\arctan 3\big)
+\frac{(-1)^n\,3^{4n+1}}{2^{2n+1}}\right).
\end{equation}
Details are left to the reader. Also notice that 
\[
\sqrt{10}\sin{\left((4n+1)\arctan{3}\right)}=10^{-2n}Im(1+3\imath)^{4n+1}.
\]
\end{remark}

The previous result extends easily to any family of polynomials $\{P_n(z)\}$ such that $P_n(z+1)-P_n(z)$ can be simplified. For example, for $a\in\mathbb{R},\,\,n\in \mathbb{N},$
the N\"{o}rlund  polynomials $\{B_n^{(a)}(x)\},$ defined through the exponential generating function 
    \begin{align*}
        \left(\frac{t}{e^{t}-1}\right)^a e^{xt} = \sum_{k=0}^{\infty} B_{n}^{(a)}(x) \frac{t^n}{n!},
    \end{align*}
    satisfy the identity
    \begin{equation}
    \label{Norlund}
        B_{n}^{(a)}(x+1) - B_{n}^{(a)}(x) = n B_{n-1}^{(a-1)}(x)
    \end{equation} 
    as a consequence of 
        \begin{align*}
         \left(\frac{t}{e^{t}-1}\right)^a e^{(x+1)t} -  \left(\frac{t}{e^{t}-1}\right)^a e^{xt} = t  \left(\frac{t}{e^{t}-1}\right)^{a-1} e^{xt}.
    \end{align*}
    In the case $a=1,$ the N\"{o}rlund polynomials coincide with the Bernoulli polynomials. 
    
    As a consequence, we have the following result.
\begin{proposition*}
    The N\"{o}rlund polynomials satisfy, for all $a\in\mathbb{R},$
    \[
    \sum_{k=0}^{n}\left(-1\right)^{k}\binom{4n+2}{4k+2}z^{4k+2}B_{4n-4k}^{(a)}
    = \frac{i}{4} (4n+2) \left[B_{4n+1}^{(a-1)}\left(-1-i\right) + B_{4n+1}^{(a-1)}\left(-i\right)- B_{4n+1}^{(a-1)}\left(i-1\right) - B_{4n+1}^{(a-1)}\left(i\right)\right].
    \]
    The specialization $a=2$ produces
    \begin{align*}
        \sum_{k=0}^{n}\left(-1\right)^{k}\binom{4n+2}{4k+2}2^{2k+1}B_{4n-4k}^{(2)} &= \frac{i(4n+2)}{4}  \left[B_{4n+1}\left(-1-i\right) + B_{4n+1}\left(-i\right)- B_{4n+1}\left(i-1\right) - B_{4n+1}\left(i\right)\right].
    \end{align*}    
\end{proposition*}
\begin{proof}
Using Proposition \ref{Prop_main}, we have 
    \begin{align*}
        \sum_{k=0}^{n}\left(-1\right)^{k}\binom{4n+2}{4k+2}z^{4k+2}B_{4n-4k}^{(a)} &=  \frac{i}{4} \left[B_{4n+2}^{(a)}\left(\frac{z}{\imath^{\frac{1}{2}}}\right) - B_{4n+2}^{(a)}\left(\frac{z}{\imath^{3/2}}\right)+ B_{4n+2}^{(a)}\left(\frac{z}{\imath^{5/2}}\right) - B_{4n+2}^{(a)}\left(\frac{z}
{i^{7/2}}\right)\right].
    \end{align*}
    With $z = \sqrt{2}$, the right-hand side is equal to 
    \begin{align*}
        &\frac{i}{4} \left[B_{4n+2}^{(a)}\left(1-i\right) - B_{4n+2}^{(a)}\left(-1-i\right)+ B_{4n+2}^{(a)}\left(-1+i\right) - B_{4n+2}^{(a)}\left(1+i\right)\right] \\
        &= \frac{i}{4} (4n+2) \left[B_{4n+1}^{(a-1)}\left(-1-i\right) + B_{4n+1}^{(a-1)}\left(-i\right)- B_{4n+1}^{(a-1)}\left(i-1\right) - B_{4n+1}^{(a-1)}\left(i\right)\right],
    \end{align*}
    where we have used \eqref{Norlund} to deduce the last expression.
\end{proof}

\subsubsection{the 4-interval formula}
\begin{figure}[h]
\centering
\includegraphics[scale=0.6]{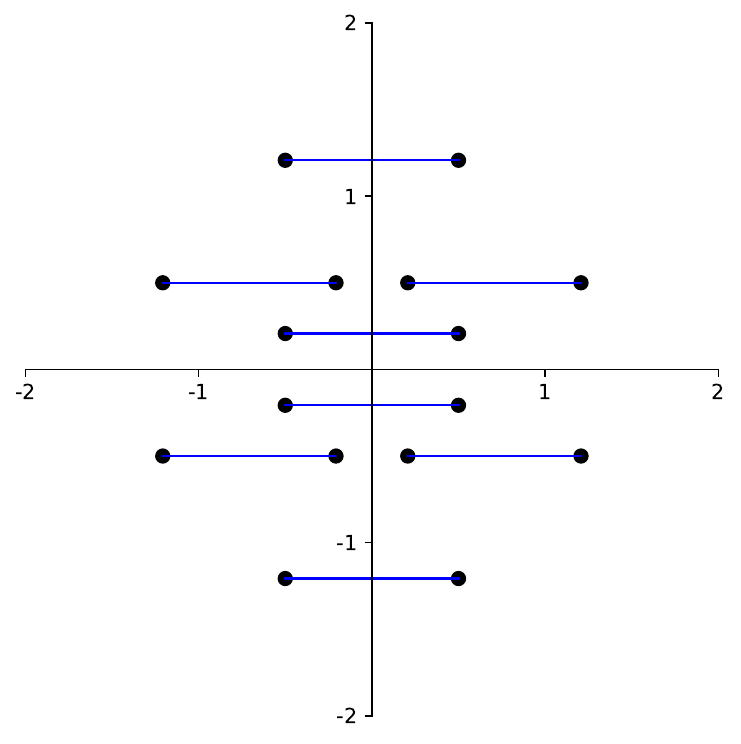} 
\caption{The 4-interval lattice.}
\label{Fig1}
\end{figure}
Following the same steps as in the previous section, and using the fact that the $4-$interval lattice shown in Fig. \ref{Fig1} is integral produces the $4$-interval formula.
\begin{proposition*}
    With $a_n=\left(1+\frac{1}{\sqrt{2}}\right)^n+\left(1-\frac{1}{\sqrt{2}}\right)^n,$ 
    \begin{align*}
    \sum_{k=0}^{n} (-1)^{k+1} \binom{8n+4}{8k+4} a_{4k+2}B_{8n-8k} = (2n+1) (-1)^{n+1}a_{4n+2}.
\end{align*}
\end{proposition*}

The proof is omitted as  it follows  the same steps as  in the case of the 2-interval formula.

\subsubsection{the 6-interval formula}
\label{233}
Our goal is to recover Ramanujan's identity \eqref{finite} using the geometric approach. Proposition \ref{Prop_main} in the case  $p=0,q=6$ produces, with $\omega = e^{\imath \frac{\pi}{6}},$
\[
\sum_{k=0}^n (-1)^k \binom{12n+6}{12k+6} z^{12k+6} B_{12n-12k}
= \frac{\imath}{12}\sum_{j=0}^{11}(-1)^j
B_{12n+6}\left(\frac{z}{\omega^{j+\frac{1}{2}}}\right).
\]
Let us  substitute $z$ successively with $2+\sqrt{3},$ $2-\sqrt{3}$
and $2$ in the left-hand side and add the three contributions. Calling the result $S_n,$ we deduce
\[
S_{n}=\frac{\imath}{12}\sum_{j=0}^{11}\left(-1\right)^{j}\left[B_{12n+6}\left(\frac{\sqrt{2+\sqrt{3}}}{\omega^{j+\frac{1}{2}}}\right)+B_{12n+6}\left(\frac{\sqrt{2-\sqrt{3}}}{\omega^{j+\frac{1}{2}}}\right)+B_{12n+6}\left(\frac{\sqrt{2}}{\omega^{j+\frac{1}{2}}}\right)\right].
\]
Noticing that
\[
\frac{\sqrt{2+\sqrt{3}}}{\omega^{\frac{1}{2}}}=\omega^{0}+\omega^{11},\frac{\sqrt{2-\sqrt{3}}}{\omega^{\frac{1}{2}}}=\omega^{2}+\omega^{9},\frac{\sqrt{2}}{\omega^{\frac{1}{2}}}=\omega^{1}+\omega^{10}
\]
and that 
\[
\frac{1}{\omega^{j}}=\omega^{11j},
\]
we deduce the expression
\[
S_{n}=\frac{\imath}{12}\sum_{j=0}^{11}\left(-1\right)^{j}\left[B_{12n+6}\left(\omega^{11j}+\omega^{11j+11}\right)+B_{12n+6}\left(\omega^{11j+2}+\omega^{11j+9}\right)+B_{12n+6}\left(\omega^{11j+1}+\omega^{11j+10}\right)\right].
\]

Figure \ref{Fig3} shows the set of the  36 points $
\{
\omega^{11j}+\omega^{11j+11}$,
$\omega^{11j+1}+\omega^{11j+10}, \omega^{11j+2}+\omega^{11j+9}\}$ with $\omega = e^{\imath \frac{\pi}{6}}$ and $0\le j\le 11.$ For convenience, these points are
now labeled $A_{0},\dots,A_{11,}$ $B_{0},\dots,B_{11}$
and $C_{0},\dots,C_{11}$, the points $A_j$ being located on the outer circle
and $C_j$ on the inner circle. A point with index $j$ has argument
$\frac{-\pi}{12}\left(j+1\right).$

\begin{figure}[h]
\centering
\includegraphics[scale=0.5]{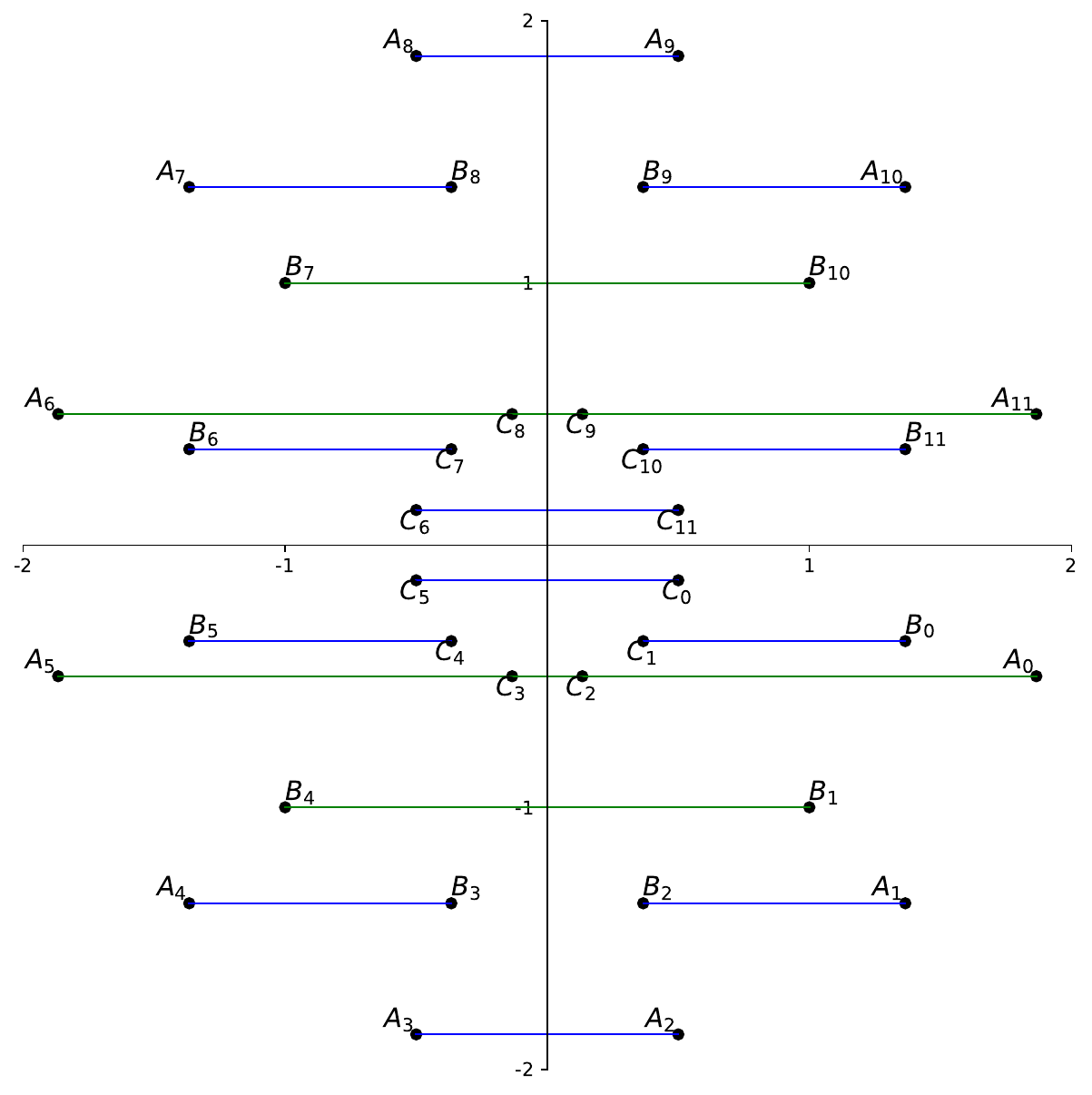}
\caption{In the 6-interval lattice, there are 18 pairs of horizontally aligned point pairs: 1 unit apart are colored blue, and pairs 2 units apart are colored green.}
\label{Fig3}
\end{figure}

The summation lattice is fine: the set of 36 points $\{A_{j},B_{j},C_{j}\}_{0\le j\le11}$ is made of 12 pairs of points $(X,Y)$ horizontally distant by $1:X=Y+1,$  namely
\[
(A_{1},B_{2}),(A_{2},A_{3}),(B_{3},A_{4}),(C_{4},B_{5}),(C_{7},B_{6}),(B_{8},A_{7}),(A_{9},A_{8}),
(A_{10},B_{9}),(B_{11},C_{10}),(B_{0},C_{1}),(C_{11},C_{6}),(C_{0},C_{5}),
\]
and $6$ pairs of points $(X,Y)$ horizontally distant by $2:X=Y+2,$  namely 
\[
(A_{0},C_{3}),(C_{2},A_{5}),(C_{9},A_{6}),(A_{11},C_{8}),(B_{10},B_{7}),(B_{1},B_{4}).
\]

Using the property \eqref{shift_Bernoulli} of Bernoulli polynomials, we deduce, for a couple of points $Y=X+1$ distant by $1$ such that $Y$ has
an even index and $X$ an odd index, 
\[
B_{12n+6}\left(X+1\right)-B_{12n+6}\left(X\right)=\left(12n+6\right)X^{12n+5}
\]
while, for a couple of points $Z=X+2$ distant by $2$ such that $Z$ has
an even index and $X$ an odd index, 
\[
B_{12n+6}\left(X+2\right)-B_{12n+6}\left(X\right)
=\left(12n+6\right)\left(X+1\right)^{12n+5}+\left(12n+6\right)X^{12n+5}.
\]
Overall, with $f\left(Z\right)=\left(12n+6\right)Z^{12n+5}$ and  $\mathcal{P}_{1}=\left\{ A_{3},A_{4},A_{7},A_{8},B_{2},B_{5},B_{6},B_{9},C_{1},C_{5},C_{6},C_{10}\right\},$ the
contribution  from the pairs of points distant by $1$ to the series $S_n$ is 
\[
\sum_{Z_{i}\in\mathcal{P}_{1}}(-1)^{i-1}f(Z_{i})
\]
while, with $g\left(Z\right)=\left(12n+6\right)\left(\left(Z+1\right)^{12n+5}+Z^{12n+5}\right)$ and $\mathcal{P}_{2}=\left\{ A_{5},A_{6},B_{4},B_{7},C_{3},C_{8}\right\},$
the contribution from the pairs of points distant by $2$  to the sum $S_n$ is 
\[
\sum_{Z_{i}\in\mathcal{P}_{2}}(-1)^{i-1}g(Z_{i}).
\]
We deduce the value  of the sum
\[
S_{n}=\frac{\imath}{12}\left(\sum_{Z_{i}\in\mathcal{P}_{1}}(-1)^{i-1}f(Z_{i})+\sum_{Z_{i}\in\mathcal{P}_{2}}(-1)^{i-1}g(Z_{i})\right),
\]
that
simplifies to 
\[
S_{n}=\frac{\imath}{12}\sum_{Z_{i}\in\mathcal{P}_{1}\cup\mathcal{P}_{2}}(-1)^{i-1}f(Z_{i}),
\]
as a consequence of the easily checked fact 
\[
\sum_{Z_{i}\in\mathcal{P}_{2}} (-1)^{i-1}g\left(Z_{i}\right)=
\sum_{Z_{i}\in\mathcal{P}_{2}}(-1)^{i-1}f\left(Z_{i}\right).
\] 
Noticing that 
\[
\mathcal{P}_{1}\cup\mathcal{P}_{2}=\left\{ A_{3},A_{4},A_{5},A_{6},A_{7},A_{8},B_{2},B_{4},B_{5},B_{6},B_{7},B_{9},C_{1},C_{3},C_{5},C_{6},C_{8},C_{10}\right\} 
\]
and evaluating the function $f$ at each point of this set produces, after simplification, the value of $S_n$ that appears on the left-hand side of \eqref{finite}.

\section{The  generating function method}
The use of the geometric method requires, except in the simple 2-and 4-interval case, the a priori knowledge of the sequence of coefficients $\{a_k\}$.
This knowledge requires, in turn, the design of a lattice that is both integral and invariant by rotation, a difficult problem.
Thankfully, the generating function method will allow us to explicitly 
produce N-interval identities
without having to compute the integral lattice. The drawback is that it will produce a sequence that is not optimal, as is explained in the next subsections. However, it  will also generate new extensions of Ramanujan's identity. Let us start with an example.

\subsection{the 6-interval case}
We now approach Ramanujan's identity \eqref{convolution} under the angle of generating functions. 
\begin{proposition*}
Ramanujan's identity \eqref{1} is of the form
\[
    6\psi(z)f(z)=zf'(z)
\]
with the generating functions
\[
\psi(z)=\sum_{n\ge 0} \frac{B_{12n}}{(12n)!}z^n
=\frac{1}{12}\sum_{j=0}^{11} 
\frac{z\omega^j}{e^{z\omega^j}-1}
\] 
and
\begin{equation}
\label{f(z)}
f(z) =
\sum_{k\ge 0}(-1)^ka_{6k+3} \frac{z^{12k+6}}{(12k+6)!}
=-\frac{16}{3}\imath 
\prod_{j=0}^{5}\sinh\left(\frac{z \omega^j }{2}\right),
\end{equation}
and with $\omega = e^{\imath \frac{\pi}{6}}$ and the sequence $\{a_n\}$  as defined in \eqref{1}.
\end{proposition*}
\begin{proof}
Using the multisection identity \eqref{general multisection} and the Bernoulli generating function \eqref{Bernoulli}, the series
\[
\psi(z)=\sum_{n\ge 0} \frac{B_{12n}}{(12n)!}z^n
\]
is first expressed as
\[
\psi(z)=\frac{1}{12}\sum_{j=0}^{11} 
\frac{z\omega^j}{e^{z\omega^j}-1}
\]
with $\omega = e^{\imath \frac{\pi}{6}}$ the twelfth-root of unity.
Next we notice that, with $f(z)$ defined as in \eqref{f(z)}, equation \eqref{convolution} assumes the form $6\psi(z)f(z)=zf'(z)$ and thus can be rewritten as
\begin{align*}
    \frac{f'(z)}{f(z)} = \frac{1}{2}\sum_{j=0}^{11} 
\frac{\omega^j}{e^{z\omega^j}-1}.
\end{align*}
Upon integrating, we obtain
\begin{align*}
    \log f(z) +\log K = \frac{1}{2} \sum_{j=0}^{11} \log(e^{-\omega^jz}-1) =  \log\left(\prod_{j=0}^{5} 2\imath \sinh\left(\frac{ \omega^j z}{2}\right) \right),
\end{align*} 
where we have used the fact that
\begin{equation}\label{sinh}
    \prod_{j=0}^{11} (e^{z\omega^j}-1) = \prod_{j=0}^{5} (e^{z\omega^j}-1) (e^{-z\omega^j}-1)  =
      \prod_{j=0}^{5} (-4)\sinh^2\left(\frac{z \omega^j }{2}\right).
\end{equation}
Computing the integration constant $K$ so that $f(z)=\frac{z^6}{12}+o(z^6),$ we deduce the solution 
\[
f(z) = -\frac{16}{3}\imath 
\prod_{j=0}^{5}\sinh\left(\frac{z \omega^j }{2}\right).
\]
It is easily checked that the sequence $\{a_n\}$  in \eqref{convolution} satisfies 
\[
\sum_{k\ge 0}(-1)^ka_{6k+3} \frac{z^{12k+6}}{(12k+6)!} = -\frac{16}{3}\imath \prod_{j=0}^{5}  \sinh\left(\frac{z \omega^j }{2}\right). 
\]
We conclude this proof by noticing that the choice of the integration constant $K$ is arbitrary since the sequence $\{a_n\}$ in \eqref{convolution} is defined up to an arbitrary non-zero multiplicative constant.
\end{proof}

\subsection{Generating  function: the general case}

\begin{proposition}\label{dim1}
With $N$ an even integer and $\omega=e^{\imath\frac{\pi}{N}}$ the $2N-$th root of unity, the $N-$ interval version of  Ramanujan's identity reads
\[
N\sum_{j\ge 0}B_{2Nj}\frac{z^{2Nj}}{(2Nj)!}
\sum_{n\ge 0}(-1)^N a_{nN+N/2}\frac{x^{2Nn+N}}{(2Nn+N)!} =
\sum_{n\ge 0}(-1)^N a_{nN+N/2}\frac{x^{2Nn+N}}{(2Nn+N-1)!} 
\]
where the generating function of the $\{a_n\}$ coefficients is
    \begin{equation}
    \label{anN}
    \sum_{n\ge 0}(-1)^n a_{nN+N/2}\frac{x^{2Nn+N}}{(2Nn+N)!} 
    = \imath \frac{2^N}{2N} 
    \prod_{j=0}^{N-1} \sinh\left(\frac{z\omega^j}{2}\right)
    \end{equation}
\end{proposition}
\begin{proof}
Proceeding as in the previous section, with the generating function 
\[\psi(z) = \frac{1}{2N} \sum_{k=0}^{2N-1} \frac{\omega^k z}{e^{\omega^k z}-1},\]
the equation 
\begin{equation}\label{eqn}
    N f(z) \psi(z) = zf'(z)
\end{equation}
admits the solution 
\[Kf(z) = \prod_{j=0}^{N-1}2\imath\sinh\left(\frac{\omega^{j}z}{2}\right).\]
Choosing the constant of integration $K$ such that $f(z) = \frac{z^N}{2N}+o(z^N)$ produces the desired result.
\end{proof}

\subsection{the coefficients  $\{a_n\}$}
In order to simplify notations, we consider  the sequence $\alpha_{k}^{\left(N\right)}$ defined by the generating function
\begin{equation}
\label{gf}
\sum_{k\ge0}\alpha_{k}^{\left(N\right)}\frac{z^{2kN+N}}{\left(2kN+N\right)!}=\prod_{j=0}^{N-1}\sinh\left(\omega^{j}z\right), 
\end{equation}
so that the sequences  $\{a_k\}$ and $\{\alpha_{k}^{\left(N\right)}\}$ are related as
\begin{equation}
\label{alpha_a}
(-1)^k a_{Nk+N/2} = \alpha_{k}^{\left(N\right)}. 
\end{equation}

We propose next several methods to evaluate the coefficients $\alpha_{k}^{\left(N\right)}$, and we relate them to classic special functions.

\subsubsection{link with  Bernoulli-Barnes polynomials}
The  expression \eqref{f(z)} for $f(z)$ suggests a link with Bernoulli-Barnes polynomials with negative order, defined by their generating function \cite[(39) p.40]{Bateman} (see also N\"{o}rlund p.142)
\begin{equation}
\label{BBgf}
\sum_{n \ge 0}B_n^{(-m)}(x; \alpha_1,\dots,\alpha_m)\frac{z^n}{n!}=e^{zx}\prod_{k=1}^m \frac{e^{\alpha_k z} -1}{\alpha_k z}.
\end{equation}
\begin{proposition*}
The coefficients $\alpha_k^{(N)}$ are related to the Bernoulli Barnes polynomials with negative order $-N$ as
\[
\alpha_n^{(N)}
=
\frac{\imath^{N}}{2N}
\frac{(2nN+N)!}{(2nN)!} 
B_{2nN}^{(-N)}
\left(\frac{1}{\omega -1};
1,\omega,\dots,\omega^{N-1}\right).
\]
\end{proposition*}
\begin{proof}
With $\omega=e^{\imath \frac{\pi}{N}}$ and  $x=- \sum_{j=0}^{N-1} \omega^j = \frac{2}{\omega-1},$ a short computation using \eqref{BBgf} produces
\[
\prod_{j=0}^{N-1} \sinh \left(\frac{z\omega^j}{2}\right)
= \imath^{N-1}\left(\frac{z}{2}\right)^N
\sum_{n\ge 6}B_{n}^{(-N)}\left(
\frac{1}{\omega-1} ; 1,\omega,\dots,\omega^{N-1}\right)\frac{z^{n}}{n!}
\]
so that, using \eqref{anN},
\[
a_{nN+\frac{N}{2}}
=
\frac{\imath^{N+2n}}{2N}
\frac{(2nN+N)!}{(2nN)!} 
B_{2nN}^{(-N)}
\left(\frac{1}{\omega -1};
1,\omega,\dots,\omega^{N-1}\right).
\]
Using \eqref{alpha_a} produces the desired result.
\end{proof}
\subsubsection{a recurrence}
Computing the coefficients $\alpha_k^{(N)}$ is a difficult task, even for low values of $N.$ The following identity  allows to compute them recursively.
\begin{theorem}
The coefficients $\alpha_n^{(2N)}$ are computed in terms of the coefficients $\alpha_n^{(N)}$ by the recurrence identity
\begin{equation}
\label{recurrence}
\alpha_{n}^{\left(2N\right)}=\imath\sum_{l=0}^{2n}\left(-1\right)^{l}\frac{\alpha_{2n-l}^{\left(N\right)}\alpha_{l}^{\left(N\right)}}{\left(4Nn-2Nl+N\right)!\left(2Nl+N\right)!}.
\end{equation}
\end{theorem}
\begin{proof}
The coefficients $\alpha_k^{(N)}$ are defined by the generating function
\[
\varphi_{N}\left(z\right)=\sum_{k\ge0}\alpha_{k}^{\left(N\right)}\frac{z^{2Nk+N}}{\left(2Nk+N\right)!}=\prod_{j=0}^{N-1}\sinh\left(ze^{\imath j\frac{\pi}{N}}\right)
\]
Now consider 
\[
\varphi_{2N}\left(z\right)=\sum_{k\ge0}\alpha_{k}^{\left(2N\right)}\frac{z^{4Nk+2N}}{\left(4Nk+2N\right)!}=\prod_{j=0}^{2N-1}\sinh\left(ze^{\imath j\frac{\pi}{2N}}\right)
\]
and use the bisection technique to obtain
\[
\varphi_{2N}\left(z\right)=\prod_{j=0}^{2N-1}\sinh\left(ze^{\imath j\frac{\pi}{2N}}\right)=\prod_{j=0}^{N-1}\sinh\left(ze^{\imath\left(2j\right)\frac{\pi}{2N}}\right)\prod_{j=0}^{N-1}\sinh\left(ze^{\imath\left(2j+1\right)\frac{\pi}{2N}}\right)
\]
\[
=\prod_{j=0}^{N-1}\sinh\left(ze^{\imath j\frac{\pi}{N}}\right)\prod_{j=0}^{N-1}\sinh\left(ze^{\imath\frac{\pi}{2N}}e^{\imath j\frac{\pi}{N}}\right)
\]
\[
=\varphi_{N}\left(z\right)\varphi_{N}\left(ze^{\imath\frac{\pi}{2N}}\right)
\]
so 
\[
\sum_{k\ge0}\alpha_{k}^{\left(2N\right)}\frac{z^{4Nk+2N}}{\left(4Nk+2N\right)!}=\sum_{k\ge0}\alpha_{k}^{\left(N\right)}\frac{z^{2Nk+N}}{\left(2Nk+N\right)!}\sum_{k\ge0}\alpha_{k}^{\left(N\right)}\frac{\left(ze^{\imath\frac{\pi}{2N}}\right)^{2Nk+N}}{\left(2Nk+N\right)!}
\]
\[
=\imath\sum_{k\ge0}\alpha_{k}^{\left(N\right)}\frac{z^{2Nk+N}}{\left(2Nk+N\right)!}\sum_{k\ge0}\alpha_{k}^{\left(N\right)}\frac{z^{2Nk+N}\left(-1\right)^{k}}{\left(2Nk+N\right)!}.
\]
Expanding
\[
\sum_{k\ge0}\alpha_{k}^{\left(N\right)}\frac{z^{2Nk+N}}{\left(2Nk+N\right)!}\sum_{k\ge0}\alpha_{k}^{\left(N\right)}\frac{z^{2Nk+N}\left(-1\right)^{k}}{\left(2Nk+N\right)!}
=\sum_{k,l}\left(-1\right)^{l}\frac{\alpha_{k}^{\left(N\right)}\alpha_{l}^{\left(N\right)}}{\left(2Nk+N\right)!\left(2Nl+N\right)!}z^{2N\left(k+l\right)+2N}
\]
\[
=\sum_{n\ge0}\sum_{l=0}^{n}\left(-1\right)^{l}\frac{\alpha_{n-l}^{\left(N\right)}\alpha_{l}^{\left(N\right)}}{\left(2Nn-2Nl+N\right)!\left(2Nl+N\right)!}z^{2Nn+2N}
=\sum_{n\ge0}\sum_{l=0}^{2n}\left(-1\right)^{l}\frac{\alpha_{2n-l}^{\left(N\right)}\alpha_{l}^{\left(N\right)}}{\left(4Nn-2Nl+N\right)!\left(2Nl+N\right)!}z^{4Nn+2N}
\]
we deduce
\[
\alpha_{n}^{\left(2N\right)}=\imath\sum_{l=0}^{2n}\left(-1\right)^{l}\frac{\alpha_{2n-l}^{\left(N\right)}\alpha_{l}^{\left(N\right)}}{\left(4Nn-2Nl+N\right)!\left(2Nl+N\right)!}.
\]
\end{proof}

\subsubsection{link with multiple zeta values}

Multiple zeta values  \cite{Zagier} are defined  as the nested zeta series
\[
\zeta(N_1,\dots,N_k) = \sum_{n_1 > \dots > n_k>0}\frac{1}{n_1^{N_1}\dots n_k^{N_k}}, 
\]
with the special case $N_1=\dots=N_k=N$ denoted as
\[
\zeta(\{N\}_{k})=\zeta(\underbrace{N,\ldots,N}_{k}).
\]
In \cite{Chen}, Y.-H. Chen cleverly relates Ramanujan's lacunary identities for Bernoulli numbers to the convolution identity for multiple zeta values due to Bowman and Bradley \cite{Bowman}:
\[
k \zeta\left(\{N\}_k\right) = \sum_{j=1}^k
(-1)^{j+1}\zeta(jN)\zeta\left(\{N\}_{k-j}\right).
\]
For $N$ even, the value $\zeta(jN)$ can be expressed as a Bernoulli number using Euler's well-known identity
\[
\zeta(2k) = (-1)^{k+1}\frac{(2\pi)^{2k}B_{2k}}{2(2k)!}.
\]
It is thus not surprising that the coefficients $\alpha_k^{(N)}$ can be expressed in terms of multiple zeta values, as follows.
\begin{proposition*}
The coefficients $\alpha_k^{(N)}$ are related to multiple zeta values as
\begin{align}
\label{mzv}
    \alpha_{k}^{(N)} = \imath \frac{(-1)^{(N+1)(k+1)}(2kN+N)!}{\pi^{2kN}}\zeta(\{2N\}_{k}).
\end{align}    
\end{proposition*}

\begin{proof}
    
Since the hyperbolic  function admits the Weirstrass infinite product representation
\begin{align*}
    \sinh(x) = \imath x\prod_{n=1}^{\infty}\left(1+ \frac{x^2}{n^2\pi^2}\right),
\end{align*}
we have 
\begin{align*}
    \prod_{j=0}^{N-1}\sinh\left(\omega^{j}z\right) &= \prod_{j=0}^{N-1}\prod_{n=1}^{\infty}\imath z\omega^j\left(1+ \frac{(z\omega^j)^2}{n^2\pi^2}\right) \\
    &= (\imath z)^N \omega^{\frac{N(N-1)}{2}}\prod_{n=1}^{\infty}\prod_{j=0}^{N-1}\left(1+ \frac{(z\omega^j)^2}{n^2\pi^2}\right) \\
    &= (\imath z)^N  \omega^{\frac{N(N-1)}{2}} \prod_{n=1}^{\infty} \left(1- (-1)^N \frac{z^{2N}}{\pi^{2N}n^{2N}}\right) \\
    & =  (-1)^{(N+1)} \imath \sum_{k=0}^{\infty} (-1)^{(N+1)k}\zeta(\underbrace{2N,\ldots,2N}_{k})\frac{z^{2kN+N}}{\pi^{2kN}}.
\end{align*}
Using \eqref{gf}, we deduce the formula 
\begin{align}\label{mzv}
    \alpha_{k}^{(N)} = \imath \frac{(-1)^{(N+1)(k+1)}(2kN+N)!}{\pi^{2kN}}\zeta(\{2N\}_{k}).
\end{align}
\end{proof}
Borwein et al. \cite{Borwein} have computed the first values of $\zeta(\{2N\}_k)$ as
\[
\zeta(\{2\}_k)=\frac{2(2\pi)^{2k}}{(2k+1)!}\left(\frac{1}{2}\right)^{2k+1},\,\,
\zeta(\{4\}_k)=\frac{4(2\pi)^{4k}}{(4k+2)!}\left(\frac{1}{2}\right)^{2k+1}
\]
\[
\zeta(\{6\}_k)=\frac{6(2\pi)^{6k}}{(6k+3)!},\,\,
\zeta(\{8\}_k)=\frac{8(2\pi)^{8k}}{(8k+4)!}\left( \left(1+\frac{1}{\sqrt{2}}\right)^{4k+2}+\left(1-\frac{1}{\sqrt{2}}\right)^{4k+2}\right)
\]
\[
\zeta(\{10\}_k)=\frac{10(2\pi)^{10k}(L_{10k+5}+1)}{(10k+5)!},\,\,
\zeta(\{12\}_k)=\frac{12(2\pi)^{12k}}{(12k+6)!}\left( \left(2-\sqrt{3}\right)^{6k+3}+\left(2+\sqrt{3}\right)^{6k+3}+2^{6k+3}\right)
\]
with $\{L_n\}$ the sequence of Lucas numbers with $L_n=L_{n-1}+L_{n-2}$ and $L_1=1,\,\,L_2=3.$

Notice that Riordan \cite[p.139]{Riordan} produces the explicit $5-$interval expression for Bernoulli numbers.
\begin{example}
For $N=2,$ $\omega=e^{\imath \frac{\pi}{2}},$
\[
\sinh z\sinh\left(\imath z\right)=\imath\sum_{k\ge0}\left(-1\right)^{k}\frac{2^{2k+1}}{\left(4k+2\right)!}z^{4k+2}
\]
so that 
\[
\alpha_{k}^{\left(2\right)}=\imath\left(-1\right)^{k}2^{2k+1},
\]
whereas, for $N=6,$ with $\omega=e^{\imath\frac{\pi}{6}},$
\[
\sinh z\sinh\left(\omega z\right)\dots\sinh\left(\omega^{5}z\right)=12\imath\sum_{k\ge0}\left(-1\right)^{k}2^{12k}\frac{\left(2-\sqrt{3}\right)^{6k+3}+2^{6k+3}+\left(2+\sqrt{3}\right)^{6k+3}}{\left(12k+6\right)!}z^{12k+6}
\]
so that 
\[
\alpha_{k}^{\left(6\right)}=12\imath\left(-1\right)^{k}2^{12k}\left[\left(2-\sqrt{3}\right)^{6k+3}+2^{6k+3}+\left(2+\sqrt{3}\right)^{6k+3}\right].
\]
These two results recover those derived previously.
\end{example}
\subsection{a generalization}
The next step is a parametric version of Ramanujan's identity that allows  us to generalize it to sequences other than Bernoullis. Notice that Lehmer \cite{Lehmer} gives lacunary identities for Euler, Genocchi and Lucas numbers.

With $0\le i\le 2^N-1$, let us denote $\epsilon^{i}=\left\{\epsilon_{0}^{i},\dots,\epsilon_{N-1}^{i}\right\}$  one of the  $2^N$ possible configurations such that $\epsilon_j^i = \pm 1$ with $0\le j\le N-1,$ and define the signature $\pi(\epsilon^i )\in \left\{-1,1\right\}$ of configuration $\epsilon^i$ as $$\pi(\epsilon^i)=\prod_{j=0}^{N-1}\epsilon_j^i.$$ 
\begin{theorem}\label{lattice_sum}
We have the polynomial identity
\begin{equation}
\label{polynomial}
\sum_{k\ge0}\alpha_{k}^{\left(N\right)}\binom{2nN+N+2p}{2kN+N}w^{2nN-2kN+2p}=
\frac{1}{2^{N}}
\sum_{i=0}^{2^N-1}
\pi\left(\epsilon^i\right)
\left(w+\sum_{j=0}^{N-1}\epsilon_j^i \omega^{j}\right)^{2nN+N+2p}.
\end{equation}
\end{theorem}
\begin{proof}
Sustituting $z$ with the differential operator  $\partial=\frac{d}{dw}$ in \eqref{gf} and applying the resulting operator to  the monomial $w^{2nN+N+2p}$ produces
\[
\prod_{j=0}^{N-1}\sinh\left(\omega^{j}\partial\right)w^{2nN+N+2p}=\sum_{k\ge0}\alpha_{k}^{\left(N\right)}\binom{2nN+N+2p}{2kN+N}w^{2nN-2kN+2p}
\]
Since, as a consequence of the expansion $\sinh{\partial}=\frac{e^{\partial}-e^{-\partial}}{2},$
\[
\prod_{j=0}^{N-1}\sinh\left(\omega^{j}\partial\right)=\frac{1}{2^{N}}
\sum_{i=0}^{2^N-1}
\pi\left(\epsilon^i\right)
\exp\left(\left(\sum_{j=0}^{N-1}\epsilon_j^i \omega^{j}\right)\partial\right),
\]
we deduce, using Taylor's formula $e^{a \partial}f(w)=f(w+a),$
\[
\sum_{k\ge0}\alpha_{k}^{\left(N\right)}\binom{2nN+N+2p}{2kN+N}w^{2nN-2kN+2p}=\frac{1}{2^{N}}
\sum_{i=0}^{2^N-1}
\pi\left(\epsilon^i\right)
\left(w+\sum_{j=0}^{N-1}\epsilon_j^i \omega^{j}\right)^{2nN+N+2p}.
\]
\end{proof}

\subsection{Applications}
A first consequence  of this generalization is another explicit expression for the coefficients $\alpha_{n}^{\left(N\right)},$ now as a multivariate integral.
\begin{proposition*}
An integral representation for the coefficients $\alpha_{n}^{\left(N\right)}$ is
\[
\alpha_{n}^{\left(N\right)}=
\frac{\left(-1\right)^{\frac{N-1}{2}}}{2^N}\frac{\left(2nN+N\right)!}{\left(2nN\right)!}\int_{-1}^{1}\dots\int_{-1}^{1}\left(\sum_{j=0}^{N-1}\omega^{j}u_{j}\right)^{2nN}du_{0}\dots du_{N-1}.
\]
For example, with $N=2,$
\[
\alpha_{n}^{\left(2\right)}=
\frac{\imath}{16}
\frac{\left(4n+2\right)!}{\left(4n\right)!}
\int_{-1}^{1}\int_{-1}^{1}\left(u_{0}+\imath u_{1}\right)^{4n}du_{0}du_{N-1}
=\left(2\imath\right)^{2n+1}.
\]
\end{proposition*}
\begin{proof}
    The multiple integral in easily computed -- directly or by induction -- as
    \[
    \frac{\prod_{j=0}^{N-1}\frac{1}{\omega^j}}{(2nN+N)!}
    \sum_{i=0}^{2^N-1}\pi(\epsilon^i)
    \left(\sum_{j=0}^{N-1}
    \epsilon_j^i \omega^j
    \right)^{2nN+N},
    \]
 with   
$ \prod_{j=0}^{N-1}\frac{1}{\omega^j} = (-1)^{\frac{N-1}{2}}.$ 
Setting $w=0$ and $p=0$ in \eqref{polynomial} produces
\[
\alpha_{n}^{\left(N\right)}
=
\frac{1}{2^{N}}
\sum_{i=0}^{2^N-1}
\pi\left(\epsilon^i\right)
\left(\sum_{j=0}^{N-1}\epsilon_j^i \omega^{j}\right)^{2nN+N}
\]
and the desired result follows by identification of both expressions.
\end{proof}

Formula \eqref{polynomial} allows, for example, an extension of Ramanujan's identity to Bernoulli polynomials.
\begin{proposition*}
The  Bernoulli polynomials  $B_n(w)$ satisfy the identity
\[
\sum_{k\ge0}\alpha_{k}^{\left(N\right)}\binom{2nN+N+2p}{2kN+N}B_{2nN-2kN+2p}(w)
\]
\[
=\frac{2nN+N+2p}{2^{N}}
\sum_{i=0}^{2^{N-1}-1}
\pi\left(\epsilon^i\right)
\left(\left(w+\sum_{j=1}^{N-1}\epsilon_j^i \omega^{j}\right)^{2nN+N+2p-1}
+\left(w-1+\sum_{j=1}^{N-1}\epsilon_j^i \omega^{j}\right)^{2nN+N+2p-1}\right).
\]
\end{proposition*}

\begin{proof}
Still using the symbolic notation $(B+w)^n$ for the Bernoulli polynomial $B_n(w),$ let us replace $w$ with $w+B$ in  identity \eqref{polynomial}, producing for the left-hand side
\[
\sum_{k\ge0}\alpha_{k}^{\left(N\right)}\binom{2nN+N+2p}{2kN+N}B_{2nN-2kN+2p}(w).
\]
With $\omega^0=1$ and $\epsilon_j^i=\pm1,$ we rewrite the right-hand side as
\[
\frac{1}{2^{N}}
\sum_{i=0}^{2^{N-1}-1}
\pi\left(\epsilon^i\right)
\left(w+1+\sum_{j=1}^{N-1}\epsilon_j^i \omega^{j}\right)^{2nN+N+2p}
+\frac{1}{2^{N}}
\sum_{i=0}^{2^{N-1}-1}
\pi\left(\epsilon^i\right)
\left(w-1+\sum_{j=1}^{N-1}\epsilon_j^i \omega^{j}\right)^{2nN+N+2p}
\]
\[
=
\frac{1}{2^{N}}
\sum_{i=0}^{2^{N-1}-1}
\pi\left(\epsilon^i\right)
\left(\left(w+1+\sum_{j=1}^{N-1}\epsilon_j^i \omega^{j}\right)^{2nN+N+2p}
-
\left(w-1+\sum_{j=1}^{N-1}\epsilon_j^i \omega^{j}\right)^{2nN+N+2p}
\right).
\]
Since $\left(B+w+1\right)^{n}-\left(B+w-1\right)^{n}=B_n(w+1)-B_n(w-1)=
n\left[w^{n-1}+\left(w-1\right)^{n-1}\right],$
we deduce
\[
\sum_{k\ge0}\alpha_{k}^{\left(N\right)}\binom{2nN+N+2p}{2kN+N}B_{2nN-2kN+2p}(w)=\frac{2nN+N+2p}{2^{N}}
\]
\[
\sum_{i=0}^{2^{N-1}-1}
\pi\left(\epsilon^i\right)
\left(\left(w+\sum_{j=1}^{N-1}\epsilon_j^i \omega^{j}\right)^{2nN+N+2p-1}
+\left(w-1+\sum_{j=1}^{N-1}\epsilon_j^i \omega^{j}\right)^{2nN+N+2p-1}\right).
\]
\end{proof}
As another application of formula \eqref{polynomial}, we propose the following result that relates a convolution of multisected Barnes zeta functions 
 \[
\zeta_{2}\left(w,p\right)=\sum_{m_{1},m_{2}\ge0}\frac{1}{\left(w+m_{1}+\imath m_{2}\right)^{p}},
\]
to the Hurwitz zeta function
 \[
\zeta\left(w,p\right)=\sum_{m\ge0}\frac{1}{\left(w+m\right)^{p}}.
\]
\begin{proposition*}
For $p>1,$  the Barnes zeta function 
satisfies the lacunary identity
    \[
\imath^p\sum_{k\ge0}(-1)^{k}\binom{4k+p+1}{p-1}2^{2k+1}\zeta_{2}\left(w,p+4k+2\right)
\]
\[
=\zeta(-\imath w,p)-(-1)^p\zeta(\imath w,p)
+
\zeta(-\imath (w-1),p)-(-1)^p\zeta(\imath (w-1),p).
\]
\end{proposition*}

\begin{proof}
We obtain first
\[
\sum_{k\ge0}(-1)^{k}\binom{4k+p+1}{p-1}2^{2k+1}x^{-p-4k-2}
\]
\[
=\frac{1}{4\imath}\left[\left(\left(x+1+\imath\right)^{-p}-\left(x-1+\imath\right)^{-p}\right)-\left(\left(x+1-\imath\right)^{-p}-\left(x-1-\imath\right)^{-p}\right)\right]
\]
as a consequence of the application of the differential operator
\[
\sin\partial\sinh\partial=\sum_{k\ge0}\left(-1\right)^{k}\frac{2^{2k+1}}{\left(4k+2\right)!}\partial^{4k+2}
\]
to the negative power $x^{-p}.$ 
Replacing $x$ with $w+m_{1}+\imath m_{2}$ and summing over $m_{1}$
and $m_{2}$ produces
\[
\sum_{k\ge0}(-1)^{k}\binom{4k+p+1}{p-1}2^{2k+1}\zeta_{2}\left(w,p+4k+2\right)=
\]
\[
\frac{1}{4\imath}\left[\zeta_{2}\left(w+1+\imath,p\right)-\zeta_{2}\left(w-1+\imath,p\right)\right]
-\frac{1}{4\imath}\left[\zeta_{2}\left(w+1-\imath,p\right)-\zeta_{2}\left(w-1-\imath,p\right)\right].
\]
But, by telescoping,
\[
\zeta_{2}\left(w+1+\imath,p\right)-\zeta_{2}\left(w-1+\imath,p\right)
\]
\[
=\sum_{m_{1}\ge1,m_{2}\ge0}\frac{1}{\left(w+m_{1}+\imath m_{2}\right)^{p}}-\sum_{m_{1}\ge-1,m_{2}\ge0}\frac{1}{\left(w+m_{1}+\imath m_{2}\right)^{p}}
=\sum_{m_{2}\ge0}\frac{1}{\left(w+\imath m_{2}\right)^{p}}+\sum_{m_{2}\ge0}\frac{1}{\left(w-1+\imath m_{2}\right)^{p}}
\]
while accordingly
\[
\zeta_{2}\left(w+1-\imath,p\right)-\zeta_{2}\left(w-1-\imath,p\right)
=\sum_{m_{2}\ge0}\frac{1}{\left(w-\imath m_{2}\right)^{p}}+\sum_{m_{2}\ge0}\frac{1}{\left(w-1-\imath m_{2}\right)^{p}}
\]
so that
\[
\sum_{k\ge0}(-1)^{k}\binom{4k+p+1}{p-1}2^{2k+1}\zeta_{2}\left(w,p+4k+2\right)=
\]
\[
=\sum_{m_{2}\ge0}\frac{1}{\left(w+\imath m_{2}\right)^{p}}+\sum_{m_{2}\ge0}\frac{1}{\left(w-1+\imath m_{2}\right)^{p}}
-\sum_{m_{2}\ge0}\frac{1}{\left(w-\imath m_{2}\right)^{p}}-\sum_{m_{2}\ge0}\frac{1}{\left(w-1-\imath m_{2}\right)^{p}}
\]
and the final result is deduced after simplification.
\end{proof}

\section{Lacunary series for Eisenstein series}
A last extension of Ramanujan's lacunary identity is now provided. It involves Eisenstein series. 

Given a lattice $\Lambda = \langle w_1,w_2\rangle$ such that $\Im\left(\frac{w_2}{w_1}\right)>0$, the corresponding Weierstra{\ss} sigma function $\sigma(z;\Lambda)$ and  Weierstra{\ss} zeta function $\zeta(z;\Lambda)$ are defined by
\begin{align*}
& \sigma(z;\Lambda) \coloneqq z \prod_{\omega \in \Lambda \setminus \{0\}}
\left(1-\frac{z}{\omega}\right)
\exp\!\left(\frac{z}{\omega}+\frac{1}{2}\left(\frac{z}{\omega}\right)^{2}\right), \\
&  \zeta(z, \Lambda) \coloneqq \frac{\sigma'(z;\Lambda)}{\sigma(z;\Lambda)} = \frac{1}{z} + \sum_{\omega \in \Lambda^{*}} \left(\frac{1}{z-\omega} +\frac{1}{\omega} + \frac{z}{\omega^2}\right).
\end{align*}
For even \(k\ge2\), let the Eisenstein series
\begin{align*}
    G_{2k}(\Lambda) = \sum_{w \in \Lambda^\ast} \frac{1}{w^{2k}}.
\end{align*}
They are the Taylor coefficients of the analytic part of the Weierstra{\ss} zeta function $\zeta(z;\Lambda)$:
\begin{align*}
    \zeta  {(z;\Lambda )}={\frac {1}{z}}-\sum _{k=1}^{\infty }{G}_{2k+2}\left(\Lambda\right)z^{2k+1}.
\end{align*}
\begin{prop}
\label{dim2}
    Let $\Lambda$ be a lattice and $N\in\mathbb{N}$.
    Then \[ N\left(1-\sum_{j\ge1}G_{2Nj}(\Lambda)\,z^{2Nj}\right)\; \sum_{n\ge0}(-1)^N a_{(2n+1)N}{(\Lambda)}\, \frac{z^{(2n+1)N}}{((2n+1)N)!} \;=\; \sum_{n\ge0}(-1)^N a_{(2n+1)N}{(\Lambda)}\, \frac{z^{(2n+1)N}}{((2n+1)N-1)!}, \] where the generating function of coefficients $a_{n}(\Lambda)$ satisfies
\[\frac{\imath z^{N}}{2N}\exp\!\left(-\frac{1}{2}\sum_{m=1}^{\infty}\frac{G_{2Nm}(\Lambda)}{m}\,z^{2Nm}\right) = \sum_{n\ge0}(-1)^N a_{(2n+1)N}{(\Lambda)}\,
\frac{z^{(2n+1)N}}{((2n+1)N)!}.\]
\end{prop}
\begin{proof}
Let $\omega=e^{\imath \frac{\pi}{N}}$ and 
\[
\Phi_N(z):=\frac{z}{2N}\sum_{k=0}^{2N-1}\omega^k\,\zeta(z\omega^k;\Lambda)
=1-\sum_{m\ge1}G_{2Nm}(\Lambda)\,z^{2Nm},
\]
and let $f$ satisfy the differential equation
\[
N\,\Phi_N(z)\,f(z)=z\,f'(z)
\]
or, equivalently,
\begin{align*}
    \frac{zf'(z)}{f(z)}= N \Phi_N(z) = \frac{z}{2}\sum_{k=0}^{2N-1} \omega^k\frac{\sigma'(z\omega^k;\Lambda)}{\sigma(z\omega^k;\Lambda)}.
\end{align*}
This equation admits the solution 
\begin{align*}
    f(z) = C\left(\prod_{k=0}^{2N-1}\sigma(z\omega^k; \Lambda)\right)^{\frac{1}{2}}.
\end{align*}
Choose the integration constant $C$ so that $f(z)=\frac{z^N}{2N}+O(z^{3N})$
and notice that 
\begin{align}\label{id1}
\prod_{k=0}^{2N-1} \sigma(z\omega^k;\Lambda) &= z^{2N} \prod_{k=0}^{2N-1} \prod_{w \in \Lambda^*} \left(1-\frac{z\omega}{w}\right)
\exp\!\left(\frac{z\omega}{w}+\frac{1}{2}\left(\frac{z\omega}{w}\right)^{2}\right) \\
&= -z^{2N} \prod_{w \in \Lambda^*} \left(1- \frac{z^{2N}}{w^{2N}}\right). \nonumber
\end{align}
Taking the exp-log transformation gives 
$$f(z) = \frac{\imath z^{N}}{2N}\exp\!\left(-\frac{1}{2}\sum_{m=1}^{\infty}\frac{G_{2Nm}(\Lambda)}{m}\,z^{2Nm}\right),$$
establishing the desired result.
\end{proof}
\begin{proposition*}
An expression for the coefficients $a_n(\Lambda)$ in terms of the Bell polynomials is 
\[
a_{(2n+1)N}(\Lambda)
= \imath\,\frac{(-1)^N}{2N}\,\frac{((2n+1)N)!}{(2N n)!}\,
B_{2N n}\!\big(b_1,\ldots,b_{2Nn}\big),
\qquad n\ge 0.
\]
\end{proposition*}
\begin{proof}
The complete Bell polynomials are defined by the exponential generating function
\[
\exp\!\Big(\sum_{k=1}^{\infty} x_k\,\frac{t^{k}}{k!}\Big)
= \sum_{n=0}^{\infty} B_n(x_1,\ldots,x_n)\,\frac{t^{n}}{n!}.
\]
Then
\[
\exp\!\left(-\frac{1}{2}\sum_{m=1}^{\infty}\frac{G_{2Nm}(\Lambda)}{m}\,z^{2Nm}\right)
= 1 + \sum_{n=1}^{\infty} \frac{z^{2N n}}{(2N n)!}\,
B_{2N n}\!\big(b_1,\ldots,b_{2Nn}\big),
\]
where
\[
b_k=
\begin{cases}
-\,N\,(k-1)!\,G_{k}(\Lambda), & N\mid k,\\[2mm]
0, & \text{otherwise}.
\end{cases}
\]
This produces the desired expression for $a_n(\Lambda)$ in terms of Bell polynomials.
\end{proof}

\begin{remark}
    Let $\Lambda = \langle 1, \tau \rangle$ with $\Im(\tau)>0$. Noticing that , as $\Im(\tau) \to \infty$ 
    \begin{align*}
        z^N \prod_{w \in \Lambda^\ast} \left(1-\frac{z^{2N}}{w^{2N}}\right)^{1/2} = z^{N}\prod_{k=1}^{\infty}\left(1-\frac{z^{2N}}{k^{2N}}\right) =\imath \frac{(-1)^N}{\pi^N}\prod_{j=0}^{N-1} \sinh(\imath \pi w^{j} z),
    \end{align*}
    we deduce 
    \begin{align*}
        \lim_{\Im(\tau) \to \infty} a_k (\Lambda) = (-1)^N \imath (\imath \pi)^{2kN+N}\alpha_{k}^{(N)}.
    \end{align*}
    Thus the limiting case recovers Proposition \ref{dim1} when $N$ is even.
\end{remark}

\section{The general lattice}

Let $\omega= e^{\frac{i \pi}{N}}$, 
let $\boldsymbol{\omega} = (1,\omega, \ldots, 
\omega^{N-1})$,
let $S(N) = \{\boldsymbol{s}=(s_1,\ldots,s_N)\mid1 \leq i \leq N, \, s_i = \pm 1\}$ and let $\pi(\boldsymbol{s}) = \prod_{i=1}^{N} s_i$. We use the notations 
$$\boldsymbol{\omega}\cdot \boldsymbol{s}   = \sum_{i=1}^{N} s_i \omega^{i-1}\,\,\text{and}\,\,\pi(\boldsymbol{s})=\prod_{i=1}^{N}s_i.$$ 
Using the product to sum formula
\begin{equation}\label{sinh}
    \prod_{j=0}^{N-1}\sinh(\omega^{j} z)
=\begin{cases}
\displaystyle \frac{1}{2^{N}}\sum_{\boldsymbol{s}\in S(N)}\pi(\boldsymbol{s})\,
\sinh\bigl(z\,(\boldsymbol{\omega}\cdot \boldsymbol{s})\bigr), & \text{\(N\) odd},\\[1.25em]
\displaystyle \frac{1}{2^{N}}\sum_{\boldsymbol{s}\in S(N)}\pi(\boldsymbol{s})\,
\cosh\bigl(z\,(\boldsymbol{\omega}\cdot \boldsymbol{s})\bigr), & \text{\(N\) even},
\end{cases}
\end{equation}
the coefficients $\alpha_{k}^{(N)}$ can be written as 
\begin{align*}
    \alpha_{k}^{(N)} = \frac{1}{2^N} \sum_{s \in S(N)} \pi(s)(\boldsymbol{\omega}\cdot \boldsymbol{s})^{2kN+N}.
\end{align*}
The sum includes $2^N$ terms. However, as illustrated in the case $N=6$ above, out of the $2^6=64$ terms, only $36$ are necessary: the others cancel by symmetry. It appears to be difficult to determine which and how many terms cancel in this sum for an arbitrary value of $N.$ However, when $N\ge3$ is a prime number, an explicit expression without redundant  terms can be found for the coefficients $\alpha_k^{(p)}$ as follows.

In order to simplify the sum in equation \eqref{sinh}, we decompose the set $S(N)$ into orbits under the action induced by taking the dot product with  vector $\boldsymbol{\omega}$. In particular, $S(N)$ admits a natural action of the cyclic group $C_{2N} = \langle c \rangle$ of order $2N$ via the map $c \cdot (s_1,\ldots,s_N) = (-s_N, s_1, \ldots, s_{N-1})$. Let $p \geq 3$ be a prime and $N=p$. Since $c^{2p}=1$, the size of an orbit divides $2p$.
It is  clear that $S(p)$ does not have orbits of size $1$ or $p$. Furthermore, the only orbit of size $2$ is $\{(1,1,\ldots,1),(-1,1,\ldots,-1)\}$. As a result, $S(p)$ admits the orbit decomposition
\[S(p) = \{(1,1,\ldots,1),(-1,1,\ldots,-1)\} \bigcup_{k=1}^{m} A_k,\]
where $m = \frac{2^{p-1}-1}{p}$. Let $\{a_i\}_{i=1}^{m}$ be  representatives of the orbits $A_i$. There are $2p$ points in $A_i$ that lie on circle of radius $|a_i|$ and are equispaced by angle $\pi/p$.

\begin{prop}\label{orbit}
    Let $p\geq 3$ be a prime number. Then there holds 
    \begin{align*}
        \alpha_{k}^{(p)} = \frac{p}{2^{p-1}} \sum_{l=1}^{m} \pi(\boldsymbol{a}_l) (\boldsymbol{\omega}\cdot \boldsymbol{a}_l)^{2kp+p},
    \end{align*}
    where $m = \frac{2^{p-1}-1}{p}$.
\end{prop}
Before proving the proposition, we prove that elements in $A_i$ correspond to distinct complex numbers upon taking dot product with $\boldsymbol{\omega}$. 


\begin{lemma}\label{1}
    Let $p\geq 3$ be a prime and let $\boldsymbol{r}, \boldsymbol{s} \in S(p)$. Then  
\[
\boldsymbol{r} \cdot \boldsymbol{\omega} = \boldsymbol{s} \cdot \boldsymbol{\omega}\]
if and only if
\[
\boldsymbol{r}=\boldsymbol{s}
\]
or
\[\boldsymbol{r},\boldsymbol{s} \in \{(1,-1,\ldots,1), (-1,1,\ldots,-1)\}.\]
\end{lemma}
\begin{proof}
    Suppose $(r_1,\ldots,r_p) \cdot \boldsymbol{\omega} = (s_1,\ldots,s_p)\cdot\boldsymbol{\omega}$. Then notice that $\sum_{i=0}^{p-1} (r_i-c_i)\omega^i =0$ and that $r_i - c_i \in \{-2,0,2\}$. Since the minimal polynomial of $e^{\frac{\imath \pi}{p}}$ is $\sum_{i=0}^{p-1} (-1)^i x^{i}$, we have $r_i - c_i = (-1)^i b$ for $b \in \{-2,0,2\}$. The desired claim follows.
\end{proof}
In other words, unlike in the case of $6$-interval formula, the only overlapping points in the lattice, when $N$ is a prime, are at the origin with multiplicity $2$. We now prove Proposition \ref{orbit}.
\begin{proof}
Lemma \ref{1} combined with orbit decomposition yields 
\begin{align*}
    \alpha_{k}^{(p)} = \frac{1}{2^p}\sum_{\boldsymbol{s} \in S(p)} \pi(s) (\boldsymbol{\omega} \cdot \boldsymbol{s})^{2kp+p} =  \frac{1}{2^p}\sum_{j=1}^{\frac{2^p-1}{p}} \sum_{\boldsymbol{a} \in A_j} \pi(\boldsymbol{a})\,(\boldsymbol{\omega}\cdot \boldsymbol{a})^{2kp+p}.
\end{align*}
Since $A_n = \{a_n \omega^{i}\, \vert\, 0 \leq i \leq 2p-1\}$, there holds 
\begin{align*}
    \sum_{\boldsymbol{a} \in A_n} \pi(a)\,(\boldsymbol{\omega}\cdot \boldsymbol{a})^{2kp+p}
    &= \pi(\boldsymbol{a}_n)\,(\boldsymbol{\omega}\cdot \boldsymbol{a}_n)^{2kp+p}
       \sum_{m=0}^{2p-1} (-1)^m\,\omega^{m(2kp+p)} = 2p \pi(\boldsymbol{a}_n)\,(\boldsymbol{\omega}\cdot \boldsymbol{a}_n)^{2kp+p}.
\end{align*}
Substituting this in the previous formula gives the desired result.

\end{proof}

\begin{ex}
    Let $p=7$. The set of points $\boldsymbol{\omega} \cdot S(7) = \{\boldsymbol{\omega} \cdot s \mid s \in S(7)\}$ is shown in figure \ref{7-interval}.
The number of distinct orbits is $\frac{2^7-1}{7} = 9$.
The $9$ possible radii  $\lvert \boldsymbol{\omega}\!\cdot\! \boldsymbol{s}\rvert$ are shown in the third column of the table below. 
\begin{table}[h]
\centering
\begin{tabular}{c c c c c}
\toprule
\# & Representative $\boldsymbol{s}$ & $\arg(\boldsymbol{\omega}\!\cdot\!\boldsymbol{s})$
   & $\lvert \boldsymbol{\omega}\!\cdot\!\boldsymbol{s}\rvert$
   & $\pi(\boldsymbol{s})\,(\boldsymbol{\omega}\!\cdot\!\boldsymbol{s})$ \\
\midrule
1 & $(+1,\,-1,\,+1,\,+1,\,-1,\,-1,\,+1)$ & $0$        & 0.890084 & $-0.890084\;+\,0.000000\,i$ \\
2 & $(+1,\,+1,\,-1,\,-1,\,+1,\,+1,\,-1)$ & $0$        & 1.109916 & $-1.109916\;+\,0.000000\,i$ \\
3 & $(-1,\,+1,\,+1,\,-1,\,+1,\,-1,\,-1)$ & $0$        & 1.603875 & $\phantom{-}1.603875\;+\,0.000000\,i$ \\
4 & $(+1,\,+1,\,-1,\,+1,\,-1,\,+1,\,-1)$ & $0$        & 2.000000 & $-2.000000\;+\,0.000000\,i$ \\
5 & $(-1,\,+1,\,+1,\,+1,\,-1,\,-1,\,-1)$ & $0$        & 2.493959 & $\phantom{-}2.493959\;+\,0.000000\,i$ \\
6 & $(+1,\,+1,\,-1,\,+1,\,+1,\,-1,\,-1)$ & $0.136968$ & 2.828427 & $-2.801938\;-\,0.386193\,i$ \\
7 & $(+1,\,+1,\,+1,\,+1,\,-1,\,-1,\,+1)$ & $0.311831$ & 2.828427 & $\phantom{-}2.692021\;+\,0.867767\,i$ \\
8 & $(+1,\,+1,\,+1,\,-1,\,+1,\,-1,\,-1)$ & $0$        & 3.603875 & $-3.603875\;+\,0.000000\,i$ \\
9 & $(+1,\,+1,\,+1,\,+1,\,-1,\,-1,\,-1)$ & $0$        & 4.493959 & $-4.493959\;+\,0.000000\,i$ \\
\bottomrule
\end{tabular}
\caption{Minimal–argument representatives for each length-$14$ orbit of $C_{14}$ acting on $S(7)$, with $\omega=e^{i\pi/7}$.}
\end{table}

\begin{figure}[h]
\label{7-interval}
\centering
\includegraphics[scale=0.6]{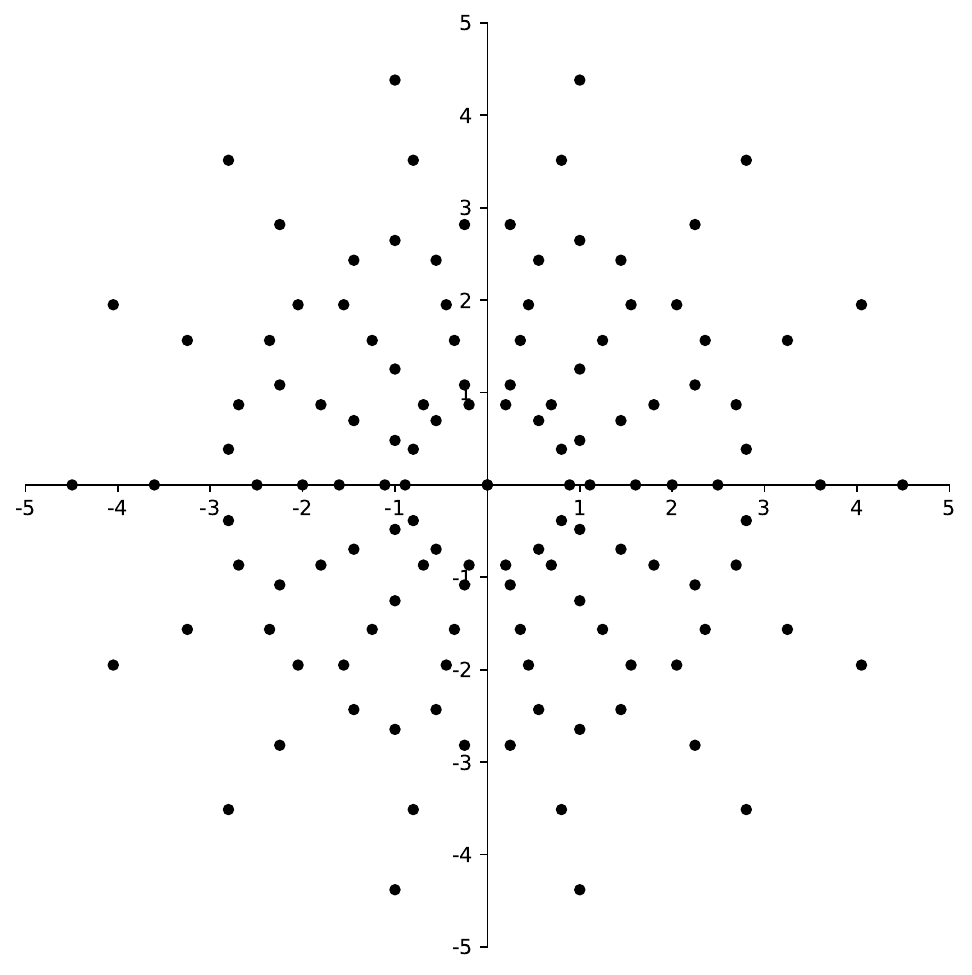} 
\caption{Lattice corresponding to $\boldsymbol{\omega}\cdot S(7)$.}
\label{Fig1}
\end{figure}
\end{ex}

\section{A probabilistic interpretation}
Ramanujan's lacunary identity for Bernoulli numbers can be given a simple probabilistic interpretation, although it is unlikely that Ramanujan considered it under this angle.
\subsection{hyperbolic secant and uniform distributions}
The Bernoulli numbers are the moments 
\[
B_{n}=\mathbb{E}\mathfrak{B}^{n},\,\,n\ge 0,
\]
of a complex-valued random variable
\[
\mathfrak{B}=\imath\mathfrak{L}-\frac{1}{2},
\]
the random variable $\mathfrak{L}$ following a square hyperbolic
secant distribution with density
\[
f_{\mathfrak{L}}=\pi\sech^{2}\left(\pi x\right).
\]
This means that we have the integral representation \cite[p.75]{Norlund}
\[
B_n = \pi \int_{-\infty}^{\infty} \left( \imath x-\frac{1}{2} \right)^n \sech^{2}\left(\pi x\right) dx.
\]
This representation was known to Ramanujan \footnote{in Ramanujan's notation, $B_n$ is the Bernoulli number with index $2n$.} under the equivalent form
\cite[Eq. 26]{Collected}
\[
\int_{0}^{\infty}\frac{x^{n}}{\left(e^{\pi x}-e^{-\pi x}\right)^{2}}dx=\frac{B_{2n}}{4\pi}.
\]
Moreover, the moment generating function of the random variable $\mathfrak{B}$
is
\[
\varphi_{\mathfrak{B}}\left(z\right)=\mathbb{E}e^{z\mathfrak{B}}=\frac{z}{e^{z}-1}.
\]
Similarly, a random variable $\mathfrak{U}$ uniformly distributed
over $\left[0,1\right]$ has moments
\[
\mathbb{E\mathfrak{U}}^{n}=\int_{0}^{1}x^n dx = \frac{1}{n+1}
\]
and moment generating function
\[
\varphi_{\mathfrak{U}}\left(z\right)=\mathbb{E}e^{z\mathfrak{U}}=\frac{e^{z}-1}{z}.
\]
The property
\[
\varphi_{\mathfrak{U}}\left(z\right)\varphi_{\mathfrak{B}}\left(z\right)=1,\,\,\forall z
\]
means that all the moments $\mathbb{E}\mathfrak{Z}^{n}$ of the complex-valued
random variable $\mathfrak{Z}\mathfrak{=B}+\mathfrak{U}$ vanish for $n\ge 1$:
\[
\mathbb{E}\left(\mathfrak{U}+\mathfrak{B}\right)^{n}=0,\,\,n\ge1.
\]
Symbolically, 
\[\mathfrak{U}+\mathfrak{B}=0,\]
a uniform random variable "cancels" a Bernoulli random variable, in the sense that all its non-zero integer moments are equal to $0.$

Now consider the multisection series, with $\omega=e^{\imath\frac{2\pi}{3}},$
\[
\sum_{n\ge0}\frac{B_{3n}}{\left(3n\right)!}z^{3n}=\frac{1}{3}\left[\varphi_{\mathfrak{B}}\left(z\right)+\varphi_{\mathfrak{B}}\left(\omega z\right)+\varphi_{\mathfrak{B}}\left(\omega^{2}z\right)\right].
\]
It can be viewed as the moment generating $\varphi_{\mathfrak{X}}\left(z\right)$  of a random variable $\mathfrak{X}$ defined as the product
\[
\mathfrak{X}=\Omega\mathfrak{B}
\]
where the complex-valued random variable $\Omega$ is independent of $\mathfrak{B}$ and uniformly distributed over the three values $\left\{ \omega^{j}\right\} _{0\le j\le2}$:
\[
\Pr\left\{ \Omega=1\right\} =\Pr\left\{ \Omega=\omega\right\} =\Pr\left\{ \Omega=\omega^{2}\right\} =\frac{1}{3}.
\]
\subsection{Ramanujan's multisection identity}
Ramanujan's lacunary identity can be interpreted probabilistically as follows: given
that a uniform random variable $\mathfrak{U}$ cancels a Bernoulli
$\mathfrak{B},$ how to cancel the more complicated random variable
$\mathfrak{X}=\Omega\mathfrak{B}$ that takes with equal probability the three values   $\mathfrak{B},$ $\omega \mathfrak{B}$ and $\omega^2 \mathfrak{B}$ ? 

Since $\mathfrak{U}$, $\omega\mathfrak{U}$ and
$\omega^{2}\mathfrak{U}$ would respectively cancel each of these possible values $\mathfrak{B},$ $\omega \mathfrak{B}$ and $\omega^2 \mathfrak{B}$
of $\mathfrak{X}$, the sum
\[
\mathfrak{Y}=\mathfrak{U}_{0}+\omega\mathfrak{U}_{1}+\omega^{2}\mathfrak{U}_{2}
\]
will cancel the Bernoulli part in $\mathfrak{X}$ while introducing a residual term, the right-hand side in Ramanujan's identity.

At the level of generating functions, this is expressed as
\[
\varphi_{\mathfrak{X}}\left(z\right)\varphi_{\mathfrak{Y}}\left(z\right)=\frac{1}{3}\left[\varphi_{\mathfrak{B}}\left(z\right)+\varphi_{\mathfrak{B}}\left(\omega z\right)+\varphi_{\mathfrak{B}}\left(\omega^{2}z\right)\right]\varphi_{\mathfrak{U}_{0}}\left(z\right)\varphi_{\mathfrak{U}_{1}}\left(\omega z\right)\varphi_{\mathfrak{U}_{2}}\left(\omega^{2}z\right)
\]
\[
=\frac{1}{3}\left[\varphi_{\mathfrak{U}_{1}}\left(\omega z\right)\varphi_{\mathfrak{U}_{2}}\left(\omega^{2}z\right)+\varphi_{\mathfrak{U}_{0}}\left(z\right)\varphi_{\mathfrak{U}_{2}}\left(\omega^{2}z\right)+\varphi_{\mathfrak{U}_{0}}\left(z\right)\varphi_{\mathfrak{U}_{1}}\left(\omega z\right)\right].
\]
This sum is easily omputed as $\sum_{n\ge 0}b_{n}\frac{z^{6n}}{(6n)!}$ with $b_{n}=\frac{2}{(6n+1)(6n+2)},$
and  proves the lacunary formula 
\[
\sum_{k=0}^n \binom{6n+3}{6k}B_{6k} = 2n+1
\]
derived by Lehmer \cite{Lehmer}.



\end{document}